\documentclass[preprint,12pt]{elsarticle}

\usepackage{amssymb}
 \usepackage{amsthm}
 \usepackage{amsmath}
\usepackage{geometry}

\usepackage{mathtools}
\usepackage[mathscr]{eucal}
\usepackage{color}
\usepackage{amsmath}
\usepackage{amssymb}
\usepackage{amsfonts}
\usepackage{amscd}
\usepackage{amsthm}
\usepackage{latexsym}
\usepackage{graphicx}

\usepackage{amsthm}

\def\be{\begin{equation}}
\def\ee{\end{equation}}
\def\bse{\begin{subequations}}
\def\ese{\end{subequations}}

\newtheorem{thm}{Theorem}
\newtheorem{lem}[thm]{Lemma}

\newdefinition{rmk}{Remark}

\begin{document}

\begin{frontmatter}

\title{Quasilinear Reaction Diffusion Systems with Mass Dissipation}
\author{Evangelos Latos\footnote{Institut f\"ur Mathematik und
    Wissenschaftliches Rechnen, Heinrichstra{\ss}e 36, 8010
    Graz, Austria. Email: \texttt{evangelos.latos@uni-graz.at}. } and
     Takashi Suzuki\footnote{Center for Mathematical Modeling and Data Science, Osaka University. Email: \texttt{suzuki@sigmath.es.osaka-u.ac.jp}}
        }

\begin{abstract}
%% Text of abstract
We study quasilinear reaction diffusion systems relative to the Shigesada-Kawasaki-Teramoto model. Nonlinearity standing for the external force is provided with mass dissipation. Estimate in several norms of the solution is provided under the restriction of diffusion coefficients, growth rate of reaction, and space dimension. 
\end{abstract}


\end{frontmatter}

%%%%%%%%%%%%%%%%%%%%%%%%%%%%%%%%%%%%%%%%%%%%%%%%%%%%%%%%%%%%%%%%%%%%%%%%%%%%%%%%%%%%%%%%%%%%%%%%%%%%%%%%%%%%%%%%%%%%%%%%%%%%%%%%%%%%%%%%%%%%%%%%%%%%%%%%%%%%%%%%%%%%%%%%%%%%%%%%%%%%%%%%%%%%%%%%%%%%%%%%%%%%%%%%%%%%%%%%%%%%%%%%%%%%%%%%%%%%%%%%%%%%%%%%%%%%%%%%%%%%%%%%%%%%%%%%%%%%%%%%%%%%%%%%%%%%%%%%%%%%%%%%

\section{Introduction}

Quasilinear reaction diffusion system is given by 
\begin{eqnarray} 
& & \tau_i\frac{\partial u_i}{\partial t}-\Delta\left(d_i(u)u_i\right) =f_i(u) \quad   
\mbox{in $\Omega\times(0,T)$} \nonumber\\ 
& & \frac{\partial}{\partial\nu}\left(d_i(u)u_i\right)=0 \qquad  \qquad \quad \ \ 
\mbox{on $\partial\Omega\times(0,T)$} \nonumber\\ 
& & \left. u_i\right\vert_{t=0}=u^0_i(x)\geq 0  \qquad \quad \qquad \mbox{in $\Omega$}  
 \label{sktLV}
\end{eqnarray}  
for $1\leq i\leq N$, where $\tau=(\tau_i)\in \mathbb{R}_+^N$, $u= (u_1(x,t), \ldots, u_N(x,t))\in \mathbb{R}^N$, $\Omega\subset{\bf R}^n$ is a bounded domain with smooth boundary $\partial\Omega$, $\nu$ is the outer unit normal vector, and $u_0=(u^0_{i}(x))\not\equiv 0$ is the initial value sufficiently smooth. For the nonlinearity it is assumed that 
\begin{equation} 
d=(d_i(u)):\overline{\mathbb{R}}^N_+\to \overline{\mathbb{R}}_+^N, \quad d_i(u)\geq c_0>0, \ 1\leq i\leq N  
 \label{asd}
\end{equation} 
is smooth, and $f=(f_i(u)): \overline{\mathbb{R}}_+^N\rightarrow \mathbb{R}^N$ is locally Lipschitz continuous  and quasi-positive: 
\begin{equation}
f_i(u_1, \cdots, u_{i-1}, 0, u_{i+1}, \cdots u_N)\geq 0, \quad u=(u_i)\geq 0, \ 1\leq i\leq N. 
 \label{qp}
\end{equation} 
We have, therefore, unique existence of a positive classical solution local in time. Our purpose is to extend this solution global in time.  This question is posed in \cite{J15}  with a positive result. 

Main assumption below is the total mass dissipation 
\begin{equation} 
\sum_i f_i(u)\leq0, \quad u=(u_i)\geq 0, 
 \label{tmd}
\end{equation} 
which implies 
\begin{equation} 
\Vert \tau \cdot u(\cdot,t)\Vert_1\leq \Vert \tau\cdot u_0\Vert_1. 
 \label{tmd2}
\end{equation} 
In the semilinear case when $d_i(u)=d_i>0$ for $1\leq i\leq N$, if $f=(f_i(u))$ is of quadratic growth rate;    
\begin{equation} 
\vert f(u)\vert \leq C(1+\vert u\vert^2), \quad u=(u_i)\geq 0,  
 \label{quad}
\end{equation} 
then $u=(u_i(x,t))\geq 0$ is uniformly bounded and hence global in time,   
\begin{equation} 
T=+\infty, \quad \Vert u(\cdot,t)\Vert_\infty \leq C. 
 \label{ubgt}
\end{equation} 
This result is a direct consequence of (\ref{tmd2}) for $n=1$, and the cases $n=2$ and $n\geq 3$ are proven by \cite{sy15, psy19} and \cite{fmt20, fmt}, respectively. For the quasi-linear case of (\ref{sktLV}), however, several tools of the latter approach require non-trivial modifications, such as regularity interpolation \cite{kan90} or  Souplet's trick \cite{sou18}. Here we examine the validity of the former approach. 

So far, global in time existence of the weak solution has been discussed in details. In \cite{Cj04,Cj06,DGJ97,GGJ03a,GJV03} it is observed that appropriate logarithmic change of variables \eqref{sktLV} can be transformed into a system with a symmetric and positive definite diffusion matrix. In \cite{Cj06}, furthermore, it is shown that  
$$
E'(t)+\mathcal{D}(t)\leq C(1+E(t)), 
$$
where 
$$
E(t)=\sum_i\int_\Omega \tau_iu_i(\log u_i-1) 
$$
and $\mathcal{D}(t)$ stands for the energy dissipation, which induces $u_i\log u_i\in L^
\infty (0,T;L^1(\Omega))$ and $\nabla\sqrt{u_i}\in L^2(\Omega_T)$. This structure is used in  \cite{CDJ18,J15}, to derive existence of the weak solution global in time to (\ref{sktLV}) for an arbitrary number of competing population species, 
\[ d_i(u)=a_{i0}+\sum_ja_{ij}u_j \] 
with non-negative and positive constants $a_{ij}$ for $1\leq i,j \leq N$ and $a_{ij}$ for $1\leq j\leq N$, respectively, under the detailed balance condition 
 \begin{equation} 
\pi_ia_{ij}=\pi_ja_{ji}, \quad 1\leq i, j \leq N 
 \label{balance}
\end{equation} 
for positive constants $\pi_i$, $1\leq i\leq N$. 

The fundametnal assumption used in this approach is 
\begin{equation} 
P=(p_{ij}(u))\geq 0, \quad u=(u_i)\geq 0 
 \label{qpositive}
\end{equation} 
for  
\begin{equation} 
p_{ij}(u)=\left( \frac{\partial d_i}{\partial u_j}+\frac{\partial d_j}{\partial u_i}\right)u_iu_j +(\delta_{ij}d_i(u)u_j+\delta_{ji}d_j(u)u_i),  
 \label{qpositive2}
\end{equation} 
where $c_0$, $\delta$, $C$ are positive constants. This assumption induces a uniform estimate of the solution in $L\log L$ norm.  

\begin{thm} 
Let $d=(d_i(u))$ satisfy (\ref{asd}) and (\ref{qpositive})-(\ref{qpositive2}). Assume that it is bounded above and below by positive constants $\delta$, $C$,  
\begin{equation} 
\delta\leq d(u)\cdot u\leq C, \quad u=(u_i)\geq 0.   
 \label{diffusion2}
\end{equation} 
Let, furthermore, $f=(f_i(u))$ satisfy (\ref{qp})-(\ref{tmd}) and is of quadratic growth rate in the sense that it satisfies (\ref{quad}) and 
\begin{equation} 
\frac{\partial f_i}{\partial u_i}\geq -C(1+\vert u\vert), \quad u=(u_i)\geq 0, \ 1\leq i\leq N.
 \label{12} 
\end{equation}  
Then, it holds that 
\begin{equation} 
\sup_{0\leq t<T}\|u(\cdot,t)\|_{L\log L}
\leq C_T  
 \label{llogl}
\end{equation}
for $u=(u_i(\cdot,t))$.  
 \label{thm3x}
\end{thm} 

\begin{thm} 
Let $d=(d_i(u))$ satisfy (\ref{asd}) and (\ref{qpositive})-(\ref{qpositive2}). Assume that it is
of linear growth rate,    
\begin{equation} 
\delta \vert u\vert^2\leq d(u)\cdot u\leq C(1+\vert u\vert^2), \quad u=(u_i)\geq 0 
 \label{diffusion}
\end{equation} 
with $\delta>0$. Assume, furthermore, the cubic growth rate of $f=(f_i(u))$: 
\begin{equation} 
\vert f_i(u)\vert \leq (1+\vert u\vert^3), \ \frac{\partial f_i}{\partial u_i}\geq -C(1+\vert u\vert^2), \quad u=(u_i)\geq 0. 
 \label{cubic2}
\end{equation} 
Then (\ref{llogl}) holds. 
 \label{cubic}
\end{thm} 

At this stage, the method of \cite{sy15, psy19} ensures $L^q$ estimate of the classical solution under the cost of low space dimension.  We require, however, an additional assumption to execute Moser's iteration \cite{ali79}.    

Letting 
\begin{equation} 
A_{ij}(u)=\frac{\partial d_i}{\partial u_j}u_i+\delta_{ij}d_i(u),
 \label{let}
\end{equation} 
we obtain 
\begin{eqnarray*} 
\frac{\partial}{\partial x_\ell}\left(d_i(u)u_i\right)
& = & \sum_j \frac{\partial d_i}{\partial u_j}\frac{\partial u_j}{\partial x_\ell}u_i+d_i(u)\frac{\partial u_i}{\partial x_\ell} \\ 
& = & \sum_j\left(
\frac{\partial d_i}{\partial u_j}u_i+\delta_{ij}d_i(u)\right)\frac{\partial u_j}{\partial x_\ell}
=\sum_jA_{ij}(u)\frac{\partial u_j}{\partial x_\ell}, 
\end{eqnarray*} 
and therefore, (\ref{sktLV}) is reduced to 
\begin{eqnarray} 
& & \tau_i\frac{\partial u_i}{\partial t}-\nabla\cdot\left(\sum_jA_{ij}(u)\nabla u_j\right) = f_i(u) \quad   \mbox{in $\Omega\times(0,T)$} \nonumber\\ 
& & \sum_jA_{ij}(u)\nabla u_j\cdot\nu=0 \qquad \qquad \qquad \qquad \mbox{on $\partial\Omega\times(0,T)$}. 
 \label{skt1}
\end{eqnarray}
The diffusion matrix $A=(A_{ij}(u))$ is not necesarily symmetric nor positive definite. Our assumption is 
\begin{equation} 
A_\alpha(u)+{}^tA_\alpha(u)\geq \delta I, \quad u=(u_i)>0, \ \alpha>0  
 \label{fundamental}
\end{equation} 
for $A_\alpha(u)=(A_{ij}^\alpha(u))$ and $A_{ij}^\alpha(u)=A_{ij}(u)(u_i/u_j)^{\alpha}$,  where $I$ denotes the unit matrix and $\delta$ is a positive constant. 

\begin{thm} 
If $f=(f_i(u))$ is of quadratic growth satisfying (\ref{quad}) and (\ref{12}). Suppose (\ref{llogl}) for the solution. Then, (\ref{fundamental}) implies 
\begin{equation} 
\sup_{0\leq t<T}\Vert u(\cdot, t)\Vert_q\leq C_T(q) 
 \label{lq}
\end{equation} 
for any $1\leq q<\infty$.   
 \label{td}
\end{thm} 

The Shigesada-Kawasaki-Teramoto (SKT) model \cite{SKT} describes separation of existence areas of competing species. There, it is assumed that $N=2$, 
\begin{align}
d_1(u)&=a_{10}+a_{11}u_1+a_{12}u_2	\nonumber\\
d_2(u)&=a_{20}+a_{21}u_1+a_{22}u_2, 
 \label{cd}
\end{align}
and 
\begin{align}
f_1(u)&=(a_1-b_1u_1-c_1u_2)u_1 
\nonumber\\
f_{2}(u)&=(a_{2}-b_2u_1-c_2u_2)u_2
 \label{clv}
\end{align}
where $a_{ij}$, $a_i$, $b_i$, $c_i$ are non-negative constants for $i,j=1,2$ and $a_{10}$, $a_{20}$ are positive constants. 

Equalities (\ref{cd}) in SKT model are due to cross diffusion where the transient probability of particle is subject to the state of the target point \cite{okubo, ss11}, while equalities (\ref{clv}) are Lotka-Volterra terms describing competition of two species in the case of 
\begin{equation} 
a_2c_1>a_1c_2, \quad a_1b_2>a_2b_1. 
 \label{a3}
\end{equation} 
The Lotka-Volterra reaction-diffusion model without cross diffusion is the semilinear case, where $d_i(u)=d_i$, $i=1,2$, are positive constants as $a_{ij}=b_{ij}=0$ in (\ref{cd}). For this system, any stable stationary solution is spatially homogeneous if $\Omega$ is convex \cite{kw85}, while  there is (non-convex) $\Omega$ which admits spatially inhomogeneous stable stationary solution \cite{mm83}.  Coming back to the SKT model, we have several results for structure of stationary solutions to a shadow system \cite{LNi, lny04, lny15, msy18}. There is also existence of the solution to the nonstationary SKT model global in time and bounded in $H^2$ norm if 
\begin{equation} 
64a_{11}a_{22}\geq a_{12}a_{21}   
 \label{yagi} 
\end{equation} 
(\cite{yag08}). Obivously, Theorems \ref{thm3x} and \ref{cubic} are not applicable to this system without total mass dissipation (\ref{tmd}).  Such $f=(f_i(u))$, admitting linear growth term in (\ref{tmd}), is called quasi-mass dissipative.  Global in time existence of the solution without uniform boundedness is the question for the general case of quasi-mass dissipation.  

We have the following theorem valid to such reaction under 
\begin{equation} 
A_\alpha(u)+{}^tA_\alpha(u)\geq 0, \quad u=(u_i)\geq 0, \ \alpha>0.  
 \label{fundamental2}
\end{equation} 

\begin{thm} 
Let $d=(d_i(u))$ satisfy (\ref{fundamental2}), and assume (\ref{qp}) and 
\begin{equation} 
f_i(u)\leq C(1+u_i), \quad u=(u_i)\geq 0, \ 1\leq i\leq N 
 \label{linear}
\end{equation} 
for $f=(f_i(u)$. Then, it holds that $T=+\infty$ for any space dimension $n$. 
 \label{thm2}
\end{thm} 

Concluding this section, we examine the condition posed in above theorems, for $d=(d_i(u))$ given by  (\ref{cd}). First, for (\ref{qpositive})-(\ref{qpositive2}), we confirm 
\begin{eqnarray*} 
& & p_{11}=2a_{10}u_1+2(2a_{11}u_1^2+a_{12}u_1u_2) \\ 
& & p_{12}=p_{21}=(a_{12}+a_{21})u_1u_2 \\ 
& & p_{22}=2a_{20}u_2+2(a_{21}u_1u_2+2a_{22}u_2^2). 
\end{eqnarray*} 
Then (\ref{qpositive2}) reads  
\[ (a_{12}+a_{21})^2u_1^2u_2^2 \geq 16(2a_{11}u_1^2+a_{12}u_1u_2)(a_{21}u_1u_2+2a_{22}u_2^2), \] 
or equivalently, 
\begin{eqnarray} 
& & \{ (a_{12}+a_{21})^2-16(a_{12}a_{21}+4a_{11}a_{22})\}u_1^2u_2^2 \nonumber\\ 
& & \quad \geq 32 (a_{11}a_{21}u_1^3u_2+a_{22}a_{12}u_1u_2^3), \quad u=(u_1, u_2)\geq 0. 
 \label{22r}
\end{eqnarray} 
Inequality (\ref{22r}) means 
\[ \{ (a_{12}+a_{21})^2-16(a_{12}a_{21}+4a_{11}a_{22})\}\geq 32(a_{11}a_{12}X+a_{22}a_{11}X^{-1}), \quad X>0 \] 
and therefore, 
\begin{equation} 
a_{11}a_{21}=a_{22}a_{12}=0, \quad (a_{12}+a_{21})^2\geq 16(a_{12}a_{21}+4a_{11}a_{22}) 
 \label{23r}
\end{equation} 
is the condition of (\ref{cd}) for (\ref{qpositive})-(\ref{qpositive2}). 

For (\ref{fundamental}), second, we note  
\begin{eqnarray} 
& & A_{11}=a_{10}+2a_{11}u_1+a_{12}u_2 \nonumber\\ 
& & A_{12}=a_{12}u_1, \ A_{21}=a_{21}u_2 \nonumber\\ 
& & A_{22}=a_{20}+a_{21}u_1+2a_{22}u_2, 
 \label{sktp}  
\end{eqnarray} 
to confirm 
\[ A_\alpha(u)=A_\alpha^0(u)+A_\alpha^1(u) \] 
for  
$A^0_\alpha(u)=\mbox{diag}(a_{10}u_1, a_{20}u_2)$ and  
\[ A^1_\alpha(u)=\left( \begin{array}{cc} 
2a_{11}u_1+a_{12}u_2 & a_{12}u_1(u_1/u_2)^\alpha \\ 
a_{21}u_2(u_2/u_1)^\alpha & a_{21}u_1+2a_{22}u_2 \end{array} \right). \] 
Hence (\ref{fundamental}) follows from $A_\alpha^1+{}^tA_\alpha^1\geq 0$, or 
\begin{eqnarray*} 
& & (a_{10}+2a_{11}u_1+a_{12}u_2)(a_{20}+a_{21}u_1+2a_{22}u_2) \\ 
& & \quad \geq \{ a_{12}u_1(u_1/u_2)^\alpha+a_{21}u_2(u_2/u_2)^\alpha\}^2, 
\end{eqnarray*}  
which is reduced to 
\[ 
(2a_{11}X+a_{12})(a_{21}X+2a_{22}) \geq \{ a_{12}X^{1+\alpha}+a_{21}X^{-\alpha}\}^2,  \quad X>0. 
\] 
This condition is thus satisfied if  
\begin{equation} 
a_{12}=a_{21}=0. 
 \label{25r}
\end{equation} 
Finally, condition (\ref{diffusion}) holds if 
\begin{equation} 
4a_{11}a_{22}\geq (a_{12}+a_{21})^2, \quad a_{11}>0, \ a_{22}>0. 
 \label{26r}
\end{equation} 
From (\ref{23r}), particularly (\ref{25r}), cross diffusion is essentially excluded in the application of Theorems \ref{cubic}, \ref{td}, \ref{thm2} to (\ref{cd}). 

\section{Proof of Theorems}\label{sec2}

We begin with the following proof. 

\begin{proof}[Proof of Theorem \ref{thm2}] 
By (\ref{skt1}) we obtain 
\begin{equation} 
\frac{\tau_i}{p+1}\frac{d}{dt}\Vert u_i\Vert_{p+1}^{p+1}+\sum_{\ell, j}\int_\Omega A_{ij}(u)\frac{\partial u_j}{\partial x_\ell}\frac{\partial u_i^p}{\partial x_\ell}=(f_i(u), u_i^{p})  
 \label{24}
\end{equation} 
for $p>0$ and $1\leq i\leq N$, and therefore,  
\begin{eqnarray*} 
\frac{1}{p+1}\frac{d}{dt}\int_\Omega \tau \cdot u^{p+1}+\sum_{ij}\int_\Omega A_{ij}(u)\nabla u_j\cdot \nabla u_i^p & = & \int_\Omega f(u)\cdot u^p \\ 
& \leq & C_1\int_\Omega \tau\cdot u^{p+1} 
\end{eqnarray*}  
by (\ref{linear}). Since 
\[  
A_{ij}(u)\nabla u_j\cdot \nabla u_i^p = \frac{4p}{(p+1)^2}A_{ij}(u)u_j^{-\frac{p-1}{2}}u_i^{\frac{p-1}{2}}\nabla u_j^{\frac{p+1}{2}}\cdot \nabla u_i^{\frac{p+1}{2}}, 
\]  
it holds that  
\begin{equation} 
\sum_{ij}A_{ij}(u)\nabla u_j\cdot \nabla u_i^p=\frac{4p}{(p+1)^2}A_{\frac{p-1}{2}}(u)[\nabla u, \nabla u].  
 \label{25xx}
\end{equation}

By (\ref{fundamental2}) we have   
\[ \frac{1}{p+1}\frac{d}{dt}\int_\Omega \tau\cdot u^p\leq C_2\left( \int_\Omega \tau\cdot u^{p+1}+1\right),  \] 
which implies 
\[ \left( \int_\Omega \tau\cdot u^{p+1}\right)^{\frac{1}{p+1}}\leq e^{C_2t}\left( \int_\Omega \tau\cdot u_0^{p+1}+1\right)^{\frac{1}{p+1}}. \] 
Then we obtain 
\[ \Vert u(\cdot,t)\Vert_\infty\leq C_3e^{C_2t}, \quad 0\leq t<T \] 
by making $p\uparrow +\infty$ with $C_3=C_3(\Vert u_0\Vert_\infty)$, and hence $T=+\infty$. 
\end{proof}

Three lemmas are needed for the proof of the other theorems. 

\begin{lem} 
Assume (\ref{qp}). Then inequality (\ref{12}) implies  
\begin{equation} 
\sum_if_i(u)\log u_i\leq C(1+\vert u\vert^2), \quad u=(u_i)\geq 0.  
 \label{28}
\end{equation} 
The second inequality of (\ref{cubic2}), similarly, implies 
\begin{equation} 
\sum_if_i(u)\log u_i\leq C(1+\vert u\vert^3), \quad u=(u_i)\geq 0.  
 \label{eqn29}
\end{equation} 
 \label{lem5} 
\end{lem} 

\begin{proof} 
The former part is proven in \cite{sy15}. The latter part follows similarly, which we confirm for completeness.  In fact, given $u=(u_i)\geq 0$, put 
\[ \tilde u_i=(u_1, \cdots, u_{i-1}, 0, u_{i+1}, \cdots, u_N). \] 
It holds that 
\begin{eqnarray} 
f_i(u) & \geq & f_i(u)-f_i(\tilde u_i) \nonumber\\ 
& = & \int_0^1\frac{\partial}{\partial s}f_i(u_1, \cdots, u_{i-1}, su_i, u_{i+1}, \cdots, u_N) \ ds  \nonumber\\ 
& = & \int_0^1\frac{\partial f_i}{\partial u_i}(u_1, \cdots, u_{i-1}, su_i, u_{i+1}, \cdots, u_N) \ ds\cdot u_i \nonumber\\ 
& \geq & -C(1+\vert u(s)\vert^2)u_i \nonumber\\ 
& \geq & -C(1+\vert u\vert^2)u_i
 \label{30}  
\end{eqnarray}  
by (\ref{qp}), where 
\[ u(s)=(u_1, \cdots, u_{i-1}, su_i, u_{i+1}, \cdots, u_N). \] 

We assume $\vert u\vert \geq 1$ because inequality (\ref{28}) is obvious for the other case of  $\vert u\vert\leq 1$. It may be also assumed that $0<s_i\leq 1$ for $u_i=s_i\vert u\vert$. Then we obtain 
\begin{eqnarray*} 
\sum_if_i(u)\log u_i & = & \sum_if_i(u)(\log \vert u\vert + \log s_i) \\ 
& \leq & \sum_if_i(u)\log s_i \\ 
& \leq & - C_4(1+\vert u\vert^2)\sum_iu_i\log s_i 
\end{eqnarray*} 
by $\vert u\vert\geq 1$, (\ref{tmd}), and (\ref{30}). It thus holds that (\ref{eqn29}) for $\vert u\vert \geq 1$ as  
\begin{eqnarray*} 
\sum_if_i(u)\log u_i & \leq  & -C_4(1+\vert u\vert^2)\vert u\vert \sum_is_i\log s_i \\ 
& \leq & C_5(1+\vert u\vert^2)\vert u\vert 
\end{eqnarray*} 
by $0<s_i\leq 1$, $1\leq i\leq N$. 
\end{proof} 

\begin{lem}
If $d=(d_i(u))$ satisfies (\ref{asd}), (\ref{tmd}), and (\ref{diffusion}), then it holds that 
\begin{equation} 
\int_0^T\Vert u(\cdot,t)\Vert_3^3 \ dt \leq C_T. 
 \label{32}
\end{equation} 
If $d=(d_i(u))$ satisfies (\ref{asd}), (\ref{tmd}), and (\ref{diffusion2}), it holds that 
\begin{equation} 
\int_0^T\Vert u(\cdot,t)\Vert_2^2 \ dt \leq C_T. 
 \label{eqn35}
\end{equation} 
 \label{lem62}
\end{lem} 

\begin{proof} 
The latter part is well-known \cite{pie10, suzuki18}.  The former part follows similarly, which we again confirm for completeness. In fact, (\ref{tmd}) implies 
\[ \frac{\partial}{\partial t}\tau\cdot u-\Delta (d(u)\cdot u)\leq 0  \ \mbox{in $\Omega\times (0,T)$}, \quad 
\left. \frac{\partial u}{\partial \nu}\right\vert_{\partial \Omega}=0 \] 
and hence 
\[ (\tau\cdot u,d(u)\cdot u)+\frac12\frac{d}{dt}\left\Vert \nabla\int_0^td(u)\cdot u\right\Vert^2_2\leq (\tau\cdot u_0,d(u)\cdot u),  \] 
where $(\cdot, \cdot)$ denotes the inner product in $L^2(\Omega)$. Then it follows that 
\begin{eqnarray*} 
\delta\min_i\tau_i\cdot \int_0^T\Vert u(\cdot, t)\Vert_3^3 \ dt & \leq & \int_0^T(\tau\cdot u,d(u)\cdot u)\;dt \\ 
& \leq & \int_0^T (\tau\cdot u_0,d(u)\cdot u)\;dt \\ 
& \leq & C\Vert \tau\cdot u_0\Vert_\infty(1+\int_0^T\Vert u(\cdot,t)\Vert_2^2 \ dt)
\end{eqnarray*} 
and hence the result. 
\end{proof} 

The following lemma has been used for construction of weak solution global in time \cite{J15, CDJ18}.   
\begin{lem}
Under the assumption of (\ref{qpositive})-(\ref{qpositive2}) it holds that 
\begin{equation}
\frac{d}{dt}\sum_i\int_\Omega \tau_iu_i(\log u_i-1) \leq \sum_i \int_\Omega f_i(u)\log u_i\;dx. 
 \label{39}
\end{equation} 
 \label{lem6}
\end{lem} 

\begin{proof}
Let 
\begin{equation} 
B=A(u)H^{-1}(u) 
 \label{5}
\end{equation}  
be the Onsager matrix, where $A=(A_{ij}(u))$ and $H(u)=diag\,(u_1^{-1},\ldots,u_N^{-1})$.  
Regard $B=B(w)$ for 
\[ w=(w_i), \quad w_i=\log u_i, \] 
and observe that (\ref{qpositive})-(\ref{qpositive2}) implies 
\begin{equation} 
B(w)+{}^tB(w)\geq 0 
 \label{37}
\end{equation} 
by (\ref{let}). We obtain, furthermore,  
\begin{eqnarray} 
& & \tau_i\frac{\partial u_i}{\partial t}-\nabla\cdot\left(\sum_jB_{ij}(w)\nabla w_j\right) =f_i(u) \quad \mbox{in $\Omega\times(0,T)$} \nonumber\\ 
& & \sum_jB_{ij}(w)\nabla w_j\cdot\nu=0 \qquad \qquad \qquad \qquad \mbox{on $\partial\Omega\times(0,T)$}   
 \label{skt2}
\end{eqnarray} 
for $1\leq i\leq N$ by (\ref{skt1}). 

Put 
\[ \Phi(s)=s(\log s-1).  \] 
Then we obtain 
\begin{eqnarray*}
\frac{d}{dt}\int_\Omega \tau\cdot\Phi(u) 
&=& 
\sum_i\int_\Omega \tau_i\frac{\partial u_i}{\partial t}\log u_i \\ 
& = & \int_\Omega f(u)\cdot w -
\sum_{i,j}B_{ij}(w)\nabla w_j\cdot \nabla w_i \ dx \\
& = & \int_\Omega f(u)\cdot w-B(w)[\nabla w,\nabla w]\;dx \\ 
& \leq & \sum_i \int_\Omega f_i(u)\log u_i\;dx
\end{eqnarray*} 
by (\ref{37}), and hence (\ref{39}).   
\end{proof} 

\begin{proof}[Proof of Theorems \ref{thm3x} and \ref{cubic}] 
These theorems are a direct consequence of Lemmas \ref{lem5}, \ref{lem62}, and \ref{lem6}. 
\end{proof}

\begin{proof}[Proof of Theorem \ref{td}] 
Any $\varepsilon>0$ admits $C_\varepsilon$ such that 
\begin{equation} 
\Vert u\Vert_1\leq \varepsilon \Vert u\Vert_{L\log L}+C_\varepsilon.  
 \label{key}
\end{equation} 
See Chapter 4 of \cite{suzuki05} for the proof. We have, on the other hand, 
\begin{eqnarray} 
\frac{1}{p+1}\frac{d}{dt}
\int_\Omega \tau\cdot u^{p+1}
+\frac{4pc_2}{(p+1)^2}
\Vert \nabla u^{\frac{p+1}{2}}\Vert_2^2 & \leq & \sum_i (f_i(u), u_i^p) \nonumber\\ 
& \leq & C(1+\Vert u\Vert_{p+2}^{p+2} ) 
 \label{scheme}  
\end{eqnarray} 
by (\ref{quad}), (\ref{fundamental}), (\ref{24}), and (\ref{25xx}), where 
\[ \nabla u^{\frac{p+1}{2}}=(\nabla u_i^\frac{p+1}{2}). \] 

Letting  
\[ z=(u_i^{\frac{p+1}{2}}), \quad r=\frac{2}{p+1}\cdot(p+2), \] 
we obtain 
\begin{equation} 
\frac{1}{p+1}\frac{d}{dt}\int_\Omega \tau\cdot u^{p+1}+\frac{c_3}{p+1}\Vert \nabla z\Vert_2^2\leq C(1+\Vert z\Vert_r^r) 
 \label{pite}
\end{equation} 
with $c_3>0$. Apply the Gagliardo-Nirenberg inequality for $n=2$, 
\begin{equation} 
\Vert z\Vert_r^r\leq C(r,q)\Vert z\Vert_q^q\Vert z\Vert_{H^1}^{r-q}, \quad 1\leq q<r<\infty. 
 \label{gni}
\end{equation} 
Here we notice Wirtinger's inequality to deduce 
\begin{eqnarray} 
\Vert u\Vert_{p+2}^{p+2} & = & \Vert z\Vert_r^r\leq C\Vert z\Vert_{H^1}^{r-1}\Vert z\Vert_1 \nonumber\\ 
 & \leq  & C (\Vert \nabla u^{\frac{p+1}{2}}\Vert_2+\Vert u\Vert_{\frac{p+1}{2}}^{\frac{p+1}{2}})^{\frac{p+3}{p+1}}\Vert u\Vert_{\frac{p+1}{2}}^{\frac{p+1}{2}}.  
  \label{iteration}  
\end{eqnarray} 
In this inequality $C$ on the right-hand side is independent of $1\leq p<\infty$, beucase it then follows that $2<r\leq 3$. 
 
For $p=1$ we use (\ref{iteration}) to derive   
\[ \Vert u\Vert_3\leq \varepsilon\Vert \nabla u\Vert_2^2+C_\varepsilon \] 
for any $\varepsilon>0$ by (\ref{key}). Then it follows that 
\begin{equation}
\sup_{0\leq t<T}\Vert u(\cdot,t)\Vert_2\leq C_T. 
 \label{2bd}
\end{equation} 
For $p>1$, second, there arises $\frac{p+3}{p+1}<1$, and hence (\ref{scheme}) and (\ref{iteration}) implies 
\begin{equation} 
\sup_{0\leq t<T}\Vert u(\cdot,t)\Vert_{\frac{p+1}{2}}\leq C_T \ \Rightarrow \ \sup_{0\leq t<T}\Vert u(\cdot,t)\Vert_{p+1}\leq C_T'.  
 \label{pbd}
\end{equation} 
By (\ref{2bd})-(\ref{pbd}) it holds that (\ref{lq}) for any $1\leq q<\infty$. 
\end{proof} 

\begin{rmk}
For system of chemotaxis in two space dimension, inequality (\ref{lq}) for $q=3$ implies uniform boundedness of the chemical term by the elliptic regulariy, which replaces the right-hand side on (\ref{scheme}) by a constant times $1+\Vert u\Vert_{p+1}^{p+1}$. Then Moser's iteration scheme induces (\ref{lq}) for $q=\infty$. See Chapter 11 of \cite{suzuki05} for details. For the case of constant $d_i$ in (\ref{sktLV}), on the other hand, the semigroup estimate is applicable as in \cite{lsy12}. If $n=2$, for example, inequality (\ref{lq}) for $q=2$ implies that for $q=\infty$. Such parabolic estimate to (\ref{sktLV}) will be discussed in future.   
\end{rmk}

%%%%%%%%%%%%%%%%%%%%%%%%%%%%%%%%%%%%%%%%%%%%%%%%%%%%%%%%%%%%%%%%%%%%%%%%%%%%%%%%%%%%%%%%%%%%%%%%%%%%%%%%%%%%%%%%%%%%%%%%%%%%%%%%%%%%%%%%%%%%%%%%%%%%%%%%%%%%%%%%%%%%%%%%%%%%%%%%%%%%%%%%%%%%%%%%%%%%%%%%%%%%%%%%%%%%%%%%%%%%%%%%%%%%%%%%%%%%%%%%%%%%%%%%%%%%%%%%%%%%%%%%%%%%%%%%%%%%%%%%%%%%%%%%%%%%%%%%%%%%%%%%%%%%%%%%%%%%%%%%%%%%%%%%%%%%%%%%%%%%%%%%%%%%%%%%%%%%%%%%%%%%%%%%%%%%%%%%%%%%%%%%%%%%%%%%%%%%%%%%%%%%%%%%%%%%%%%%%%%%%%%%%%%%%%%%%%%
%%%%%%%%%%%%%%%%%%%%%%%%%%%%%%%%%%%%%%%%%%%%%%%%%%%%%%%%%%%%%%%%%%%%%%%%%%%%%%%%%%%%%%%%%%%%%%%%%%%%%%%%%%%%%%%%%%%%%%%%%%%%%%%%%%%%%%%%%%%%%%%%%%%%%%%%%%%%%%%%%%%%%%%%%%%%%%%%%%%%%%%%%%%%%%%%%%%%%%%%%%%%%%%%%%%%%%%%%%%%%%

\section*{Acknowledgement} 
The first author is supported by the Austrian Science Fund (FWF): F73 SFB LIPID HYDROLYSIS. The second author is supported by JSPS core-to-core research project, Kakenhi 16H06576, and Kakenhi  19H01799.

\end{document}